\documentclass{amsart}

\usepackage{color,graphicx,amssymb,latexsym,amsfonts,txfonts,amsmath,amsthm}
\usepackage{pdfsync}
\usepackage{amsmath,amscd}
\usepackage[all,cmtip]{xy}

\usepackage{tkz-graph}
\usepackage{pgfplots}

\usepackage{hyperref}
\hypersetup{
    colorlinks=true,       
    linkcolor=blue,          
    citecolor=blue,        
    filecolor=blue,      
    urlcolor=blue           
}

\newtheorem{theo}{Theorem}[section]
\newtheorem{coro}[theo]{Corollary}

\theoremstyle{remark}
\newtheorem{remark}{\bf Remark}

\begin{document}

\title{Uniqueness of Generalized Fermat Groups in positive characteristic}

\author[R. A. Hidalgo]{Rub\'en A. Hidalgo}
\address{Departamento de Matem\'atica y Estad\'istica, Universidad de la Frontera\\ Casilla 54-D, Temuco \\Chile}
\email{ruben.hidalgo@ufrontera.cl}

\author[H. F. Hughes]{Henry F. Hughes}
\address{Instituto de Ciencias Físicas y Matemáticas, Facultad de Ciencias, Universidad Austral de Chile\\  Valdivia\\ Chile. Departamento de Matem\'atica y Estad\'istica, Universidad de la Frontera\\ Casilla 54-D, Temuco \\Chile}
\email{henry.hughes@uach.cl}

\author[M. Leyton-\'Alvarez]{Maximiliano Leyton-\'Alvarez}
\address{Instituto de Matem\'atica y F\'isica, Universidad de Talca\\
3460000 Talca, Chile}
\email{leyton@inst-mat.utalca.cl}

\thanks{Partially supported by the projects Fondecyt 1230001, 1220261, and 1221535}




\begin{abstract}
Let $X\subset {\mathbb P}_{K}^{m}$ be a smooth irreducible projective algebraic variety of dimension $d$, defined over an algebraically closed field $K$ of characteristic $p>0$.
We say that $X$ is a generalized Fermat variety of type $(d;k,n)$, where $n \geq d+1$ and $k \geq 2$ is relatively prime to $p$, if there is a Galois branched covering $\pi\colon X\to {\mathbb P}_{K}^{d}$, with deck group ${\mathbb Z}_k^n\cong H<\rm{Aut}(X)$, whose branch divisor consists of $n+1$ hyperplanes in general position (each one of branch order $k$). In this case, the group $H$ is called a generalized Fermat group of type $(d;k,n)$. We prove that, if $k-1$ is not a power of $p$ and either (i) $p=2$ or (ii) $p>2$ and $(d;k,n) \notin \{(2;2,5), (2;4,3)\}$, then a generalized Fermat variety of type $(d;k,n)$ has a unique generalized Fermat group of that type.

\end{abstract}

\maketitle

\section{Introduction}
Let $X\subset {\mathbb P}_{K}^{m}$ be a smooth irreducible projective algebraic variety of dimension $d \geq 1$, defined over an algebraically closed field $K$ of characteristic $p>0$. Let $\rm{Aut}(X)$ be its group of all regular automorphisms and $\rm{Lin}(X)$ be its group of linear automorphisms (that is, its automorphisms obtained as the restriction of a projective linear transformation of ${\mathbb P}_{K}^{m}$).

Let $n \geq d+1$ and $k \geq 2$ be integers such that $k$ is relatively prime to $p$.
A group  ${\mathbb Z}_k^n\cong H<\rm{Aut}(X)$ is called a generalized Fermat group of type $(d;k,n)$ if there 
is an abelian Galois branched covering $\pi\colon X\to {\mathbb P}_{K}^{d}$, with deck group $H$, whose branch divisor consists of $n+1$ hyperplanes in general position (each one of branch order $k$).
In this case, we say that $X$ (respectively, that $(X,H)$) is a generalized Fermat variety (respectively, generalized Fermat pair) of type $(d;k,n)$.

In \cite{HKLP16}, it was observed that every generalized Fermat variety of type $(1;k,n)$ (also called a generalized Fermat curve of type $(k,n)$) has a unique generalized Fermat group of type $(1;k,n)$ under the assumption that $(k-1)(n-1)>2$ and $k-1$ is not a power of $p$.
This paper shows that a similar uniqueness result holds for $d \geq 2$.

\begin{theo}\label{maintheo}
Let $d \geq 2$, $n \geq d+1$ and $k \geq 2$ be relatively prime to $p$ such that $k-1$ is not a power of $p$. If either (i) $p=2$ or (ii) $p>2$ and $(d;k,n) \notin \{(2;2,5), (2;4,3)\}$, then a generalized Fermat variety of type $(d;k,n)$ has a unique generalized Fermat group of that type.
\end{theo}

The above result, for the case of algebraically closed fields of characteristic zero, was obtained in \cite{HHL23}.

As a consequence of Pardini's classification of abelian branched covers \cite{Par91}, and that of maximal branched abelian covers \cite{AlPa13},
 it can be observed that each generalized Fermat pair $(X,H)$ of type $(d;k,n)$ is, up to isomorphisms, uniquely determined by a collection $\Sigma_1,\cdots,\Sigma_{n+1}\subset{\mathbb P}_K^d$ of hyperplanes in general position. Let us denote by $[x_{1}:\cdots:x_{d+1}]$ the points of the  projective space ${\mathbb P}_{K}^{d}$. 
 Up to the projective linear group ${\rm PGL}_{d+1}(K)$, we may assume that, for $j=1,\ldots,d+1$, the hiperplane $\Sigma_{j}$ corresponds to the hyperplane $x_{j}=0$, and that $\Sigma_{d+2}$ corresponds to the hiperplane $x_{1}+\cdots+x_{d+1}=0$. With this normalization, one can write down an explicit algebraic model $X^{k}_{n}(\Lambda) \subset {\mathbb P}_{K}^{n}$ for $X$ (as a certain fiber product of $(n-d)$ Fermat hypersurfaces of degree $k$ and dimension $d$), where $\Lambda$ belongs to certain parameter space $X_{n,d} \subset K^{d(n-d-1)}$ (see Section \ref{Section Algebraic Model}). As a consequence of the results in \cite{Kon02,MaMo64}, we may observe, under the hypothesis of Theorem \ref{maintheo}, that ${\rm Aut}(X^{k}_{n}(\Lambda))={\rm Lin}(X^{k}_{n}(\Lambda))$  (Corollary \ref{Corollary Kontogeorgis igualdad de automorfismos en GFV}). The following results state the form of these automorphisms.

\begin{theo}\label{maintheo2}
Let $d \geq 2$, $n \geq d+1$ and $k \geq 2$ be relatively prime to $p$ such that $k-1$ is not a power of $p$. If either (i) $p=2$ or (ii) $p>2$ and $(d;k,n) \notin \{(2;2,5), (2;4,3)\}$, then 
${\rm Aut}(X^{k}_{n}(\Lambda))={\rm Lin}(X^{k}_{n}(\Lambda))$ consists of matrices such that only an element in each row and column is non-zero.
\end{theo}
The method used to prove that ${\rm Lin}(X^{k}_{n}(\Lambda))$ consists of matrices such that only an element in each row and column is non-zero (see Section \ref{Section demostraciones de los teoremas principales}), was introduced in \cite{Kon02} to compute the automorphism group of Fermat-like varieties, it has also been used in \cite{Poo05} to exhibit explicit equations for hypersurfaces with trivial automorphism group, in \cite{OY19} to compute automorphism groups of quintic threefolds, and in \cite{GLMV23} to compute the automorphism groups of most Fermat, Delsarte and Klein hypersurfaces.

\begin{remark}[A connection to the field of moduli]
Fields of moduli were first introduced by Matsusaka \cite{Matsusaka} and later by Shimura \cite{Shimura} (for the case of polarized abelian varieties). In \cite{Koizumi}, Koizumi gave a more general definition for general algebraic varieties (even with extra structures). 
Let us consider an algebraic model $X^{k}_{n}(\Lambda)$, where $\Lambda \in X_{n,d}$. If $\sigma \in {\rm Aut}(K)$ (the group of the field automorphisms of $K$), then $(X^{k}_{n}(\Lambda))^{\sigma}=X^{k}_{n}(\Lambda^{\sigma})$, where $\Lambda^{\sigma} \in X_{n,d}$ is the tuple obtained from $\Lambda$ after applying $\sigma$ to each of its coordinates. The field of moduli of $X^{k}_{n}(\Lambda)$ (respectively, of $\Lambda$) is the fixed field of the subgroup of ${\rm Aut}(K)$ consisting of those $\sigma$ such that $(X^{k}_{n}(\Lambda))^{\sigma}$ is isomorphic to $X^{k}_{n}(\Lambda)$ (respectively,
the fixed field of the subgroup of ${\rm Aut}(K)$ consisting of those $\sigma$ such that there is a transformation $T_{\sigma} \in {\rm PGL}_{d+1}(K)$ sending the collection $\{\Sigma_{1}, \ldots, \Sigma_{n+1}\}$ onto the collection $\{\Sigma_{1}^{\sigma}, \ldots, \Sigma_{n+1}^{\sigma}\}$; which we denote it by the symbol $T_{\sigma}(\Lambda)=\Lambda^{\sigma}$).
Now, if we are under the hypothesis of Theorem \ref{maintheo}, we have the uniqueness of the generalized Fermat group $H \cong {\mathbb Z}_{k}^{n}$. This uniqueness ensures that $(X^{k}_{n}(\Lambda))^{\sigma}$ is isomorphic to $X^{k}_{n}(\Lambda)$ if and only if there is an element $T_{\sigma} \in {\rm PGL}_{d+1}(K)$ such that $T_{\sigma}(\Lambda)=\Lambda^{\sigma}$.
 This fact permits us to observe that the field of moduli of $X^{k}_{n}(\Lambda)$ is the same as the field of moduli of $\Lambda$. Elsewhere, we plan to use this to compute explicitly the field of moduli of 
$X^{k}_{n}(\Lambda)$ and to find explicit examples for which this is not a field of definition.
\end{remark}

\section{Preliminaries}
In the following, $K$ will denote an algebraically closed field of characteristic $p>0$, and all the algebraic varieties are assumed to be defined over it.

\subsection{Pardini's Theorem}
Let $Y$ be a smooth projective algebraic variety, $X$ be a normal projective irreducible algebraic variety, $G<$ Aut$(X)$ be a finite abelian group, and assume that $p$ does not divide the order of $G$. Let $\pi\colon X\to Y$ be an abelian cover with deck group $G$, that is to say, for every $g\in G$ it holds that $\pi\circ g=\pi$ and $Y=X/G$.

Let us denote by $G^*$ the \textbf{group of characters} of $G$, that is
$$G^*=\{\chi\ |\ \chi\colon G\to \mathcal{S}^1_K\ \text{homomorphism of groups}\},$$
where 
$$\mathcal{S}^1_K:=\{x\in K\ |\ x^j=1\text{ for some }j\in\{1,2,\cdots\}\}.$$ Since $G$ is a finite abelian group, we have that $G^*\cong G.$
We denote by ``1" the trivial character, that is $1(g)=1$ for all $g\in G$.
Let $\mathcal{O}_X$ be the sheaf of regular functions on $X$. By \cite{Be84}, the action of $G$ induces a splitting:
$$\pi_*\mathcal{O}_X=\bigoplus_{\chi\in G^*}\mathcal{L}_\chi^{-1}$$
where $G$ acts on $\mathcal{L}_\chi^{-1}$ via the character $\chi.$ The invariant summand $\mathcal{L}_1$ is isomorphic to $\mathcal{O}_Y$.
Let $R$ (respectively, $D$) be the ramification locus (respectively, the branch locus of $\pi$), that is to say, $R$ consist of the points of $X$ that have a nontrivial stabilizer and 
$\pi(R)=D$.
Let $T$ be an irreducible component of $R$. Then the \textbf{inertia group} of $T$ is defined as
$$G_T=\{h\in G\ |\ h(x)=x,\ \forall x\in T\}.$$

If $\mathfrak{C}$ denotes the set of cyclic subgroups of $G$ and, for all $\mathcal{Z}\in\mathfrak{C}$, $S_\mathcal{Z}$ is the set of generators of $\mathcal{Z}^*$ (the group of characters of ${\mathcal Z}$), then we may write:
$$D=\sum_{\mathcal{Z}\in\mathfrak{C}}\ \sum_{\psi\in S_\mathcal{Z}}D_{\mathcal{Z},\psi}$$
where $D_{\mathcal{Z},\psi}$ is the sum of all irreducible components of $D$ that have inertia group $\mathcal{Z}$ and character $\psi.$

The sheaves $\mathcal{L}_\chi,\chi\in G^*$, and the divisors $D_{\mathcal{Z},\psi}$ will be called the \textbf{building data} of the cover $\pi\colon X\to Y.$

Let $\chi_1,\cdots,\chi_s\in G^*$ be such that the group $G^*$ is the direct sum of the cyclic subgroups generated by $\chi_1,\cdots,\chi_s$. Write $\mathcal{L}_j$ for $\mathcal{L}_{\chi_j}$. We will call $\mathcal{L}_j,\ j=1,\cdots,s,\ D_{\mathcal{Z},\psi}$ a set of \textbf{reduced building data} for the cover $\pi\colon X\to Y$.

\begin{theo}[\cite{Par91}]\label{Prop 2.1 Pardini} Let $X,Y,\pi\colon X\to Y,\ G,\ \chi_1,\cdots,\chi_s,\ \mathcal{L}_1,\cdots,\mathcal{L}_s,\ D_{\mathcal{Z},\psi}$ be as before. Let $d_j$ be the order of $\chi_j$ and $\chi_{j}|_\mathcal{Z}=\psi^{r^j_{\mathcal{Z},\psi}},\ 0\leq r^j_{\mathcal{Z},\psi}<|\mathcal{Z}|,\ j=1,\cdots,s$. Then the following relations hold for the reduced building data of $\pi\colon X\to Y:$
\begin{equation}\label{equivalencia de Pardini}
d_j\mathcal{L}_j\equiv \sum_{\mathcal{Z},\psi}\dfrac{d_jr^j_{\mathcal{Z},\psi}}{|\mathcal{Z}|}D_{\mathcal{Z},\psi},\quad j=1,\cdots,s
\end{equation}
(notice that $\displaystyle \dfrac{d_jr^j_{\mathcal{Z},\psi}}{|\mathcal{Z}|}$ is an integer).

Conversely, given invertible sheaves $\mathcal{L}_j,\ j=1,\cdots,s,$ and divisors $D_{\mathcal{Z},\psi} $ such that \eqref{equivalencia de Pardini} holds, then it is possible to construct a cover $\pi\colon X\to Y$ with deck group $G$, determined uniquely up to isomorphism of cover with deck group $G$, and building data $\mathcal{L}_j,D_{\mathcal{Z},\psi}$.
\end{theo}

\begin{remark}
The Proposition 2.1 of \cite{Par91} is more general than the Theorem \ref{Prop 2.1 Pardini}, since we assume that $X$ be a normal projective irreducible algebraic variety and $Y$ be a smooth projective algebraic variety, whereas in \cite{Par91} requires $X$ to be a normal algebraic variety and $Y$ to be a smooth and complete algebraic variety.
\end{remark}

\subsection{Equivalence of generalized Fermat pairs}
We use Pardini's result to obtain the following rigidity observation.

\begin{theo}\label{Theo FGP isomorphism} Let $(X_1,H_1)$ and $(X_2,H_2)$ be generalized Fermat pairs of the same type $(d;k,n)$ and corresponding regular branched covering
$$\pi_j\colon X_j\to{\mathbb P}^d_K,\quad j=1,2.$$
If $\pi_1$ and $\pi_2$ have the same branching locus divisor
$$D=\Sigma_1+\cdots+\Sigma_{n+1},$$
then there exists $\phi\colon X_1\to X_2$ isomorphism such that 
$$\pi_1=\pi_2\circ\phi.$$
\end{theo}
\begin{proof}
Let us consider a generalized Fermat pair of type $(d;k,n)$: $(X,H)$.\\ We have $H\cong{\mathbb Z}_k^n$ and exists a cover map $\pi\colon X\to{\mathbb P}^d_K$ whose branch divisor is 
$$D=\Sigma_1+\cdots+\Sigma_{n+1},$$
where $\Sigma_1,\cdots,\Sigma_{n+1}$ are hyperplanes in general position.

Let us consider a standard set of generators of $H,\ \{a_1,\cdots,a_{n+1}\}$, i.e.,
$$a_1\cdots a_{n+1}=1,\quad H=\langle a_1,\cdots,a_n\rangle.$$

Let us denote by $U_k\subset K\setminus\{0\}$ the set of $k-$roots of 1; which is a cyclic group of order $k$ with the operation of multiplication (i.e. $U_k\cong{\mathbb Z}_k$). Let $\omega_k\in U_k$ be a generator of $U_k$, i.e., $U_k=\langle\omega_k\rangle.$\\\\
Let $H^*$ the group of characters of $H$, then
$$H^*\cong H\cong {\mathbb Z}_k^n.$$
A set of generators of $H^*$ is
$$\chi_1,\cdots,\chi_n,$$
where
$$\chi_j(a_i)=\begin{cases}\omega_k&;\ \ i=j\\1&;\ \ i\in\{1,\cdots,n\}\setminus\{j\}\\\omega_k^{-1}&;\ \ i=n+1\end{cases}$$

Let
\begin{align*}
z_j=&\ \langle a_j\rangle,\quad j=1,\cdots,n+1.\\
z_j^*=&\ \text{group of characters of }z_j=\langle\psi_j\rangle\cong{\mathbb Z}_k,\text{ where }\\
\psi_j(a_j)=&\ \omega_k.
\end{align*}
We have, for $i=1,\cdots,n,\ j=1,\cdots,n+1$
$$\chi_i\Big|_{z_j}\in z_j^*\Rightarrow\ \chi_i\Big|_{z_j}=\psi^{r_{i,j}},$$
where
$$r_{i,j}=\begin{cases}1&;\ \ i=j\\0&;\ \ i\in\{1,\cdots,n\}\setminus\{j\}\end{cases}$$
if $j\neq n+1,$ and
$$r_{i,n+1}=k-1.$$

Note that by the Theorem \ref{Prop 2.1 Pardini}, using that 
\begin{itemize}
\item $X$ smooth,
\item $Y={\mathbb P}^d_K$ smooth,
\item $G=H$, the deck group of $\pi$,
\item $r^j_{H,\psi}=r_{i,j}$,
\item $d_j=k$,
\end{itemize}
we conclude that the abelian cover $\pi\colon X\to{\mathbb P}^d_K$ is uniquely determined, up to isomorphisms, by the building data:
\begin{itemize}
\item $\chi_1,\cdots\chi_n$: uniquely determined by $(k,n)$,
\item $\mathcal{L}_{\chi_1},\cdots,\mathcal{L}_{\chi_n}$: there are our variables,
\item $\Sigma_1,\cdots,\Sigma_{n+1}:$ the given support of our branch divisors.
\end{itemize}

The line bundles $\mathcal{L}_{\chi_1},\cdots,\mathcal{L}_{\chi_n}$ must satisfy the relations:
\begin{equation}\label{equivalencia nuestro caso}
k\mathcal{L}_{\chi_i}\equiv \sum_{j=1}^{n+1}r_{i,j}\Sigma_j=\Sigma_i+(k-1)\Sigma_{n+1}.
\end{equation}
These last equivalences are those that appear in \eqref{equivalencia de Pardini}, applied to our case.

As the relations, from \eqref{equivalencia nuestro caso}, ``determine"  the line bundles (up to linear equivalence), we obtain that $(X,H)$ is uniquely determined by $(d;k,n)$ and $\Sigma_1,\cdots,\Sigma_{n+1}$ and this proves the theorem.
\end{proof}

\begin{coro}
Let $(X_1,H_1)$ and $(X_2,H_2)$ be generalized Fermat pairs of the same type $(d;k,n)$. 
There is an isomorphism $\phi\colon X_1\to X_2$ such that 
$$\phi H_1\phi^{-1}=H_2,$$
if and only if, there exists an isomorphism  $\psi\colon {\mathbb P}^d_K\to {\mathbb P}^d_K$ such that it carries the branch locus of $\pi_1\colon X_1\to{\mathbb P}^d_K$ to that of $\pi_2\colon X_2\to{\mathbb P}^d_K$. 
\end{coro}

\subsection{The parameter space $X_{n,d}$.}\label{Subsection The parameter space}
If $d \geq 1$, then there is a natural one-to-one correspondence between hyperplanes and points in ${\mathbb P}_{K}^{d}$. This correspondence comes from the fact that each hyperplane $\Sigma \subset {\mathbb P}^{d}_{K}$ (with projective coordinates $[x_{1}:\cdots:x_{d+1}]$) has equation of the form 
$$\Sigma(\rho)=\{\rho_{1}x_{1}+\cdots+\rho_{d+1}x_{d+1}=0\}, \mbox{where }  \rho=[\rho_{1}:\cdots:\rho_{d+1}] \in {\mathbb P}_{K}^{d}.$$

A collection of hyperplanes $\Sigma(q_1),\cdots,\Sigma(q_{n+1}) \subset {\mathbb P}_{K}^{d}$, where $n \geq d+1$, is in general position if:
\begin{enumerate}
\item for $d=1$, they are pairwise different;
\item or $d\geq2$, they are pairwise different, and any $3\leq s\leq d+1$ of these hyperplanes intersect at a $(d-s)-$dimensional plane (where negative dimension means the empty set).
\end{enumerate}

In the above situation, the divisor $D$, formed by the hyperplanes $\Sigma(q_1),\cdots,\Sigma(q_{n+1})\subset {\mathbb P}^{d}_{K}$, is a strict normal crossing divisor.

\begin{remark}\label{Remark condicion de posicion general}
Note that the hyperplanes $\Sigma(q_1),\cdots,\Sigma(q_{n+1})\subset{\mathbb P}^{d}_{K}$ are in general position if the corresponding (pairwise different) points $q_1,\cdots,q_{n+1}\in{\mathbb P}^{d}_{K}$ are in general position. This means that, for $d\geq2,$ any subset $3\leq s\leq d+1$ of these points spans a $(s-1)-$plane $\Sigma\subset{\mathbb P}^{d}_{K}.$
If $q_j=[\rho_{1,j}:\cdots:\rho_{d+1,j}]\in{\mathbb P}^d_K,$ for $j=1,\cdots,n+1$, then the hyperplanes $\Sigma(q_1),\cdots,\Sigma(q_{n+1})$ are in general position if and only if, all $(d+1)\times (d+1)-$minors of the matrix
$$M(q_1\cdots,q_{n+1}):=\left(\begin{matrix}
\rho_{1,1}&\cdots&\rho_{1,n+1}\\\vdots&\vdots&\vdots\\\rho_{d+1,1}&\cdots&\rho_{d+1,n+1}
\end{matrix}\right)\in M_{(d+1)\times(n+1)}(K)$$
are nonzero.
\end{remark}

Given an ordered collection $(\Sigma(q_1),\cdots,\Sigma(q_{n+1}))$, where $n\geq d+1$, of hyperplanes in ${\mathbb P}^{d}_{K}$  in general position, there is a unique projective linear automorphism $T\in$ PGL$_{d+1}(K)$ such that $T(\Sigma(q_j))=\Sigma(e_j)$, for $j=1,\cdots,d+2$, where $e_1:=[1:0:\cdots:0],\cdots,e_{d+1}:=[0:\cdots:0:1]$ and $e_{d+2}:=[1:\cdots:1]$. In this case, for $j=d+3,\cdots,n+1$, one has that $T(\Sigma(q_j))=\Sigma(\Lambda_{j-d-2})$, where
$$\Lambda_i:=[\lambda_{i,1}:\cdots:\lambda_{i,d}:1]\in{\mathbb P}^d_K,\ \text{for }\ i=1,\cdots,n-d-1.$$
For $j=1,\cdots,d$, we set $\lambda_j:=\ (\lambda_{1,j},\cdots,\lambda_{n-d-1,j})\in K^{n-d-1},\ \Lambda:=(\lambda_1,\cdots,\lambda_d)$ and
$$\Sigma_j(\Lambda):=\begin{cases}
\Sigma(e_j)&;\text{ if }j=1,\cdots,d+2\\
\Sigma(\Lambda_{j-d-2})&;\text{ if }j=d+3,\cdots,n+1.
\end{cases}
$$
In this way, the tuples $(\Sigma(q_1),\cdots,\Sigma(q_{n+1}))$ and $(\Sigma_1(\Lambda),\cdots,\Sigma_{n+1}(\Lambda))$ are PGL$_{d+1}(K)-$ equivalent.

A parameter space $X_{n,d}$ of such tuples, up to the action of PGL$_{d+1}(K)$, can be described as follows.\\
For $n=d+1$, $X_{d+1,d}:=\{(1,\overset{d}{\cdots},1)\}$ and for $n\geq d+2$,  
$$X_{n,d}:=\{\Lambda\in K^{d(n-d-1)}\ |\ \Sigma_j(\Lambda),\ j=1,\cdots,n+1,\text{ are in general position}\},$$
where $K^{d(n-d-1)}=K^{n-d-1}\times\overset{d}{\cdots}\times K^{n-d-1}$. We observe that, in this case, $X_{n,d}$ is a non-empty set.

\subsection{Algebraic models of generalized Fermat pairs}\label{Section Algebraic Model}
Let $d \geq 1$, $n\geq d+1$, and $k\geq 2$ which is relatively prime to $p$ (the characteristic of the algebraically closed field $K$). 
Let $\Lambda=(\lambda_1,\cdots,\lambda_d)\in X_{n,d}$, where $\lambda_j=(\lambda_{1,j},\cdots,\lambda_{n-d-1,j})\in K^{n-d-1}.$
\begin{equation}\label{Equation Algebraic Model}
X_n^k(\Lambda):=\left\{\begin{matrix}
x_1^k+\cdots+x_d^k+x_{d+1}^k+x_{d+2}^k&=&\ 0\\
\lambda_{1,1}x_1^k+\cdots+\lambda_{1,d}x_d^k+x_{d+1}^k+x_{d+3}^k&=&\ 0\\
\vdots&\vdots&\ \vdots\\
\lambda_{n-d-1,1}x_1^k+\cdots+\lambda_{n-d-1,d}x_d^k+x_{d+1}^k+x_{n+1}^k&=&\ 0
\end{matrix}\right\}\subset{\mathbb P}_K^n.
\end{equation}

Let $H_0:=\langle\varphi_1,\cdots,\varphi_{n+1}\rangle$, where
$$\varphi_j([x_1:\cdots:x_j:\cdots:x_{n+1}]):=[x_1:\cdots:\omega_kx_j:\cdots:x_{n+1}],$$
and $\omega_k$ is a primitive $k-$th root of unity. Let $\rm{Fix}(\varphi_j)\subset X_n^k(\Lambda)$ be the set of fixed points of $\varphi_j$ and $F(H_0):=\cup_{j=1}^{n+1}\rm{Fix}(\varphi_j)$.

The following facts can be deduced from the above.
\begin{itemize}
\item $H_0\cong{\mathbb Z}_n^k$.
\item $\varphi_1\varphi_2\cdots\varphi_{n+1}=1.$
\item $H_0<\rm{Aut}(X_n^k(\Lambda))<{\rm{PGL}_{n+1}(K)}.$
\end{itemize}
\begin{theo}
If $\Lambda\in X_{n,d}$, then $X_n^k(\Lambda)$ is an irreducible nonsingular complete intersection.
\end{theo}
\begin{proof}
Let us observe that the matrix of coefficients of \eqref{Equation Algebraic Model} has rank $(m+1)\times(m+1)$ and all of its $(m+1)\times(m+1)-$minors are different from zero (this is the general condition of the $n+1$ hyperplanes). The result follows from \cite{Te88}[Proposition 3.1.2]. For explicitness, let us workout the case $d=2$ (the case $d=1,\ K={\mathbb C}$ was noted in \cite{GHL09}).

Set $\lambda_{0,1}=\lambda_{0,2}=1$ and consider the following degree $k$ homogeneous polynomials $f_i:=\lambda_{i,1}x_1^k+\lambda_{i,2}x_2^k+x_3+x_{4+i}^k\in K[x_1,\cdots,x_{n+1}]$, where $i\in\{0,1,\cdots,n-3\}$. Let $V_i\subset {\mathbb P}^n_K$ be the hypersurface given as the zero locus of $f_i$.

The algebraic set $X_n^k(\lambda_1,\lambda_2)$ is the intersection of the $(n-2)$ hypersurfaces $V_0,\cdots,V_{n-3}$. We consider the matrix of $\nabla f_i$, written as rows.
\begin{equation}\label{nabla f_i}
\left(\begin{array}{ccccccc}
kx_1^{k-1}&kx_2^{k-1}&kx_3^{k-1}&kx_4^{k-1}&0&\cdots&0\\
\lambda_{1,1}kx_1^{k-1}&\lambda_{1,2}kx_2^{k-1}&kx_3^{k-1}&0&kx_5^{k-1}&\cdots&0\\
\vdots&\vdots&\vdots&\vdots&\vdots&\vdots&\vdots\\
\lambda_{n-3,1}kx_1^{k-1}&\lambda_{n-3,2}kx_2^{k-1}&kx_3^{k-1}&0&\cdots&0&kx_{n+1}^{k-1}
\end{array}\right)
\end{equation}
Recall that $k$ and $p$ are relatively prime, so $k\neq0$. By the defining equations of the curve, and as $(\lambda_1,\lambda_2)\in X_{n,2}$, we see that a point which has three variables $x_i=x_j=x_l=0$ for $i\neq j\neq l\neq i$ and $1\leq i,j,l\leq n+1$ has also $x_t=0$ for $t=1,\cdots,n+1.$ Therefore the above matrix has the maximal rank $n-2$ at all points of the surface. So the defining hypersurfaces are intersecting transversally and the corresponding algebraic surface they define is a nonsingular complete intersection.

Next, we proceed to see that the ideal $I_{k,n},$ defined by the $n-2$ equations defining $X_n^k(\lambda_1,\lambda_2)\subset {\mathbb P}_K^n$ is prime. This follows similarly as in \cite{Kon02}. Observe first that defining equations $f_0,\cdots,f_{n-3}$ form regular sequence, that $K[x_1,\cdots,x_{n+1}]$ is a Cohen-Macauley ring and that the ideal $I_{k,n}$ they define is of codimension $n-2$. The ideal $I_{k,n}$ is prime as a consequence of the Jacobian Criterion \cite{Eis95}[Th. 18.15], \cite{Kon02}[Th. 3.1] and that the fact that the matrix \eqref{nabla f_i} is of maximal rank $n-2$ on $X_n^k(\lambda_1,\lambda_2)$. In \cite{Kon02}[Remark 3.4], it is pointed out that an ideal $I$ is prime if the singular locus of the algebraic set defined by $I$ has big enough codimension (in our case, the singular set is the empty set).
\end{proof}

\begin{theo}\label{Theo (X(L),H_0) is GFP} $(X_n^k(\Lambda),H_0)$ is a generalized Fermat pair of type $(d;k,n).$
\end{theo}
\begin{proof}
The map
$$\pi_0\colon X_n^k(\Lambda)\to{\mathbb P}^{d}_{K}\colon[x_1:\cdots:x_{n+1}]\to[x_1^k:\cdots:x_{d+1}^k]$$
is a regular branched cover with deck group $H_0$ and whose branch set is the union of the hyperplanes $\Sigma_1(\Lambda),\cdots, \Sigma_{n+1}(\Lambda)$, each one of order $k$. In other words, the pair $(X_n^k(\Lambda),H_0)$ is a generalized Fermat pair of type $(d;k,n)$.
\end{proof}
\begin{remark}\label{Fix(H_j) is a GFV} If we set the subgroups $H_j:=\langle\varphi_1,\cdots,\varphi_{j-1},\varphi_{j+1},\cdots,\varphi_n\rangle,\ 1\leq j\leq n$, and $H_{n+1}:=\langle\varphi_1,\cdots,\varphi_{n-2},\varphi_{n+1}\rangle$, then $H_j\cong{\mathbb Z}_k^{n-1}$. For $d\geq 2,\ H_j<\rm{Aut}(F_j)$ and there is a regular map $\pi_j\colon F_j\to \Sigma_j(\Lambda)={\mathbb P}^{d-1}_K$ (given by the restriction of $\pi_0$), which is a Galois branched cover with deck group $H_j$, such that its branch locus is given by the collection of the $n$ intersections $\pi(F_j)\cap\pi(F_i),\ i\neq j$, which are copies of ${\mathbb P}^{d-2}_K$ in general position. In particular, $(F_j,H_j)$ is a generalized Fermat pair of type $(d-1;k,n-1)$. This permits to study of generalized Fermat varieties from an inductive point of view.
\end{remark}


\begin{theo}[Fermat generalized pair isomorphism]\label{Theo principal FGP isomorphism}
Let $(X,H)$ be a generalized Fermat pair of type $(d;k,n)$ with abelian cover $\pi\colon X \to {\mathbb P}_K^d$ and branch divisor
$D=\Sigma_1+\cdots+\Sigma_{n+1}$. Then there exists $\Lambda:=(\lambda_1,\cdots,\lambda_d)\in X_{n,d}$ and $T\in\rm{PGL}_{d+1}(K)$ such that
\begin{enumerate}
\item $T(\Sigma_j(\Lambda))=\Sigma_j$, for all $j\in\{1,\cdots,n+1\}$, and 
\item there is an isomorphism $\phi\colon X_n^k(\Lambda)\to X$ such that $\pi_0=\pi\circ\phi.$
\end{enumerate}
In particular, $H_0=\phi^{-1}H\phi$.
\end{theo}
\begin{proof}
This is a direct consequence of Theorem \ref{Theo (X(L),H_0) is GFP} and Theorem \ref{Theo FGP isomorphism}.
\end{proof}


\subsection{Discussion of equality of automorphisms}\label{Section Discussion of equality of automorphisms}
Let $X$ be a smooth irreducible projective algebraic variety, defined over $K$, and $d=$dim$(X)$.
\begin{itemize}
\item The \textbf{Picard group} of $X$, denoted by ${\rm Pic}(X)$, is the group of isomorphism classes of invertible sheaves (or line bundles) on $X$, with the group operation being a tensor product.
\item The \textbf{canonical sheaf} of $X$ is
$$\omega_X:=\bigwedge^d\Omega_{X/K},$$
the $d-$th exterior power of the sheaf $\Omega_{X/K}$ of differentials. It is an invertible sheaf on $X$.
\end{itemize}

It is known that $\rm{Lin}(X)$, the subgroup of linear automorphisms of $X$, can be a proper subgroup of the group of biregular automorphisms of $X$, $\rm{Aut}(X)$, (see \cite{Se44}).

The following theorem ensures that these two groups of automorphisms are equal, under certain hypotheses. It was first seen, partially, in \cite{MaMo64}.
\begin{theo}[\cite{Kon02}]\label{Theo Kontogeorgis igualdad de automorfismos}
If $X$ is a complete intersection of dimension $d\geq 3$, or a non-singular complete intersection of dimension 2, such that $\omega_X$ is not the identity in ${\rm Pic}(X)$. Then we have that $\rm{Aut}(X)=\rm{Lin}(X)$. 
\end{theo}

In the following corollary, we will show the result of the previous theorem applied to the generalized Fermat varieties.

\begin{coro}\label{Corollary Kontogeorgis igualdad de automorfismos en GFV}
Let $d\geq 2,\ k\geq 2$ and $n\geq d+1$ be integers such that $k-1$ is not a power of $p$, and $p$ and $k$ relatively prime. Let $\Lambda\in X_{n,d}\subset K^{n-2}$. If either (i) $p=2$ or (ii) $p>2$ and
$(d;k,n)\not\in\{(2;2,5),(2;4,3)\}$, then 
$$\rm{Aut}(X_n^k(\Lambda))=\rm{Lin}(X_n^k(\Lambda)).$$
\end{coro}
\begin{proof}
If $d=2$, then $X_n^k(\Lambda)\subset {\mathbb P}_K^n$ is a complete intersection of $n-2$ hypersurfaces of degree $k$. We use the result in \cite{Har77}[exer. 8.4, p.188] and we have the following
$$\omega_{X_n^k(\Lambda)}=\mathcal{O}_{X_n^k(\Lambda)}(r),$$
where $r=(n-2)k-n-1.$ Following Theorem \ref{Theo Kontogeorgis igualdad de automorfismos}, $\rm{Aut}(X_n^k(\Lambda))=\rm{Lin}(X_n^k(\Lambda))$ holds for $r\neq0$.

If $r=0$, then $n\neq 2$ and
$$k=\dfrac{n+1}{n-2}=1+\dfrac{3}{n-2}.$$
Then $(n-2)$ is a positive integer divisor of $3$, i.e., $n=3$ or $n=5.$ In particular, we prove the statement for $d=2.$

If $d\geq3$, then $\rm{Aut}(X_n^k(\Lambda))=\rm{Lin}(X_n^k(\Lambda))$ (see Theorem \ref{Theo Kontogeorgis igualdad de automorfismos}).

On the other hand, we can notice that $(d;k,n)\in\{(2;2,5),(2;4,3)\}$ does not hold when $p=2$, since, by hypothesis, $p$ must be relatively prime to $k\in\{2,4\}$.
\end{proof}

\begin{remark}
The only types of the form $(2;k,n)$ where $\rm{Aut}(X_n^k(\Lambda))\neq\rm{Lin}(X_n^k(\Lambda))$ are the types $(2;2,5)$ and $(2;4 ,3)$. 
If $(d;k,n)=(2;2,5)$, then $\rm{Lin}(X_5^2(\Lambda))$ is a finite extension of ${\mathbb Z}_2^5$ (generically a trivial extension) and $\rm{Aut}(X_5^2(\Lambda))$ is infinite by results due to Shioda and Inose in \cite{ShiIno} (in \cite{Vi83} computed it for a particular case).
If $(d;k,n)=(2;4,3)$, then $X_3^4(\Lambda)$ is the classical Fermat hypersurface of degree 4 in ${\mathbb P}^3_{K}$ for which $\rm{Lin}(X_3^4(\Lambda))\cong{\mathbb Z}_4^3\rtimes\mathfrak{S}_4$ and $\rm{Aut}(X_3^4(\Lambda))$ is infinite.
In these two situations, the generalized group $H$ (which is unique in $\rm{Lin}(X_n^k(\Lambda))$) cannot be a normal subgroup of $\rm{Aut}(X_n^k(\Lambda))$. Otherwise, every element of $\rm{Aut}(X_n^k(\Lambda))$ will induce an automorphism of ${\mathbb P}^2_{K}$ (permuting the $n+1$ branched lines), so a finite group linear automorphism. This will ensure that $\rm{Aut}(X_n^k(\Lambda)) $ is a finite extension of $H$, so a finite group, a contradiction.
\end{remark}

\section{Proof of Theorems \ref{maintheo} and \ref{maintheo2}}\label{Section demostraciones de los teoremas principales}
We may assume that $(X,H)=(X_{n}^{k}(\Lambda),H_{0})$.

\begin{enumerate}
\item We start by showing that every element of ${\rm Lin}(X_{n}^{k}(\Lambda))$ consists of matrices such that only an element in each row and column is non-zero.

Let $\sigma\in\rm{Lin}(X_n^k(\Lambda))$. As it is linear, it is given in terms of an $(n+1)\times(n+1)$ matrix $A_\sigma=(a_{i,j})\in$ GL$_{n+1}(K)$:
$$\sigma([x_1:\cdots:x_{n+1}])=\left[\sum_{j=1}^{n+1}a_{1,j}x_j:\cdots:\sum_{j=1}^{n+1}a_{n+1,j}x_j\right].$$
The arguments are independent of $d\geq 1.$ In \cite{HKLP16}, it was done the case $d=1.$ Below, we indicates its simple modification for $d=2$ (the general case is similar).

Set $\lambda=(\lambda_1,\cdots,\lambda_{n-3}),\ \mu=(\mu_1,\cdots,\mu_{n-3})$ such that $\Lambda=(\lambda,\mu)\in X_{n,2}$. If we set (by setting $\lambda_0=\mu_0=1$)
$$f_j=\ \lambda_jx_1^k+\mu_jx_2^k+x_3^k+x_{4+j}^k,\ \ j=0,\cdots,n-3,$$
then $X_n^k(\Lambda)=V(f_0,\cdots,f_{n-3})$. In this way, for every point $P \in X_n^k(\Lambda)$, then $\sigma(P) \in X_n^k(\Lambda)$, that is, $f_i\circ\sigma=\sigma^*(f_i)\in\langle f_0,\cdots,f_{n-3}\rangle$, i.e.
$$f_i\circ\sigma=\sum_{j=0}^{n-3}g_{j,i}f_j,$$
for some constants $g_{j,i}\in K.$

Let us consider $\sigma^*(\nabla f_i)=\nabla f_i\circ \sigma$. For every point $P \in X_n^k(\Lambda)$ we have 
\begin{equation}\label{sigma(nabla (f_i sigma)) 1}
\nabla(f_i\circ\sigma)(P)=\sum_{j=0}^{n-3}g_{j,i}\nabla f_j(P).
\end{equation}
Since $\sigma$ is linear, by chain rule, we have the following
\begin{equation}\label{sigma(nabla (f_i sigma)) 2}
\nabla(f_i\circ\sigma)(P)=\nabla(f_i)(\sigma(P))A_\sigma.
\end{equation}

We use \eqref{sigma(nabla (f_i sigma)) 1} and \eqref{sigma(nabla (f_i sigma)) 2}, so
$$\sigma^*(\nabla f_i)A_\sigma=\nabla(f_i\circ\sigma)=\sum_{j=0}^{n-3}g_{j,i}\nabla f_j,$$
that is,
\begin{equation}\label{sigma*(nabla f_i)}
\sigma^*(\nabla f_i)=\left(\sum_{j=0}^{n-3}g_{j,i}\nabla f_j\right)A_\sigma^{-1}.
\end{equation}

Recall that $f_j=\lambda_jx_1^k+\mu_jx_2^k+x_3^k+x_{4+j}^k$ for $0\leq j\leq n-3$, where $\lambda_0=\mu_0=1$, and
$$Y_j:=\nabla f_j=(k\lambda_jx_1^{k-1},k\mu_jx_2^{k-1},kx_3^{k-1},0,\cdots,0,kx_{4+j}^{k-1},0,\cdots,0),$$
where the fourth nonzero element is at the $4+j$ position. So
$$\sigma^*(Y_j)=k\left(\lambda_j\mathcal{B}_1\ ,\ \mu_j\mathcal{B}_2\ ,\ \mathcal{B}_3\ ,\ 0\ ,\cdots\ ,\ 0\ ,\ \mathcal{B}_{4+j}\ ,\ 0\ ,\ \cdots\ ,0\right),$$
where 
$$\mathcal{B}_\ell=\left(\sum_{i=1}^{n+1}a_{\ell,i}x_i\right)^{k-1}=\big(\sigma^*(x_\ell)\big)^{k-1},\ \ \text{for}\ \ell\in\{1,2,3,4+j\}.$$

Now, for $t\in\{1,2,3,4+j\}$, the $t-$th coordinate of $\sigma^*(Y_j)$ involves all combinations of the terms
$$\binom{k-1}{\nu_1,\cdots,\nu_{n+1}}(a_{t,1}^{\nu_1}\cdots a_{t,n+1}^{\nu_{n+1}})\cdot(x_1^{\nu_1}\cdots x_{n+1}^{\nu_{n+1}}),\ \ \text{for }\nu_1+\cdots+\nu_{n+1}=k-1,$$
where
$$\binom{k-1}{\nu_1,\cdots,\nu_{n+1}}=\dfrac{(k-1)!}{\nu_1!\nu_2!\cdots\nu_{n+1}!}.$$

Equation \eqref{sigma*(nabla f_i)} asserts that each $\sigma^*(Y_j)$ is a linear combination of $Y_0,\cdots,Y_{n-3}$. So, it only involves combinations of the monomials $x_i^{k-1}$. The above asserts, as long as that $\binom{k-1}{\nu_1,\cdots,\nu_{n+1}}\neq0$, that in the tuple $\overline{\nu}=(\nu_1,\cdots,\nu_{n+1})$ there is only one coordinate $\nu_i=k-1$ and all the others are equal to zero.
 We set
\begin{align*}
\mathbf{x}^{\overline{\nu}}=&\ x_1^{\nu_1}\cdots x_{n+1}^{\nu_{n+1}}, \quad
A_{l,\overline{\nu}}=\ a_{l,1}^{\nu_1}\cdots a_{l,n+1}^{\nu_{n+1}}.
\end{align*}

In this notation, we have that if the tuple $\overline{\nu}$ has at least two non-zero coordinates then $\mathbf{x}^{\overline{\nu}}$ does not appear as a term in the linear combination of $Y_i$. Observe that if $\binom{k-1}{\nu_1,\cdots,\nu_{n+1}}\neq 0$, then using equation \eqref{sigma*(nabla f_i)}, we have
$$(A_{1,\overline{\nu}},\cdots,A_{n+1,\overline{\nu}})\cdot A_{\sigma}=0.$$
And, since $A_\sigma$ is an invertible matrix, the above implies that
$$A_{l,\overline{\nu}}=0,\ \ l=1,\cdots,n+1.$$

Recall that $p>0$ is the characteristic of the algebraically closed field $K$. From \cite{Eis95}[p.352], we have that $\binom{k-1}{\nu}$ is not divisible by $p$ if and only if $v_i\leq k_i$ for all $i$, where $\nu=\sum v_ip^i,\ k-1=\sum k_ip^i$ are the $p-$adic expansions of $\nu$ and $k-1$, respectively. Since by hypothesis $k-1$ is not a power of $p$, then there exists $u\in\{1,\cdots,k-2\}$ such that $\binom{k-1}{u}\neq 0.$

Let us assume that, for some $j$ and some $1\leq l_1<l_2\leq n+1$, we have $a_{j,l_1}\neq 0\neq a_{j,l_2}$. In this case, we consider the tuple $\overline{\nu}=(0,\cdots,0,u_1,0,\cdots,0,u_2,0,\cdots,0)$, where $u_1:=u\in\{1,\cdots,k-1\},\ u_2:=k-1-u,$ such that $\binom{k-1}{u}\neq 0$ and $u_\ell$ is in the position $l_\ell$ for $\ell=1,2$. As $\mathbf{x}^{\overline{\nu}}$ cannot appear in the linear combination of $Y_j$, we must have that 
\begin{align*}
0=&\ \binom{k-1}{\overline{\nu}}\cdot A_{j,\overline{\nu}}
=\ \binom{k-1}{u}\cdot a_{j,l_1}^{u_1}a_{j,l_2}^{u_2}
\Rightarrow 0=\ a_{j,l_1}^{u_1}a_{j,l_2}^{u_2},
\end{align*}
from which 
$a_{j,l_1}=0 \text{ or }a_{j,l_2}=\ 0$, a contradiction. Therefore, $A_\sigma$ is a matrix such that only an element in each row and column is non-zero.

\item Next, we prove that $H_{0}$ is the unique generalized Fermat group of type $(d; k, n)$ in ${\rm Lin}(X_{n}^{k}(\Lambda))$.

Let us assume that $H<\rm{Lin}(X_n^k(\Lambda))$ is another generalized Fermat group of $X_n^k(\Lambda)$ of type $(d;k,n)$ and let $\varphi_1^*,\cdots,\varphi_{n+1}^*$ be an standard set of generators of $H$. Recall that the only non-trivial elements of $H$ with a set of fixed points of the maximal dimension $d-1$ are the non-trivial powers of these standard generators. If $\varphi_j^*$ is given by a diagonal matrix, then it is an element of $H_0$. Let us assume that one of the standard generators, say $\varphi_1^*,$ is non-diagonal. By the previous item, the matrix representation of $\varphi_1^*$ has exactly one non-zero element in each of its arrows and each of its columns. This means that the locus of fixed points $\hat{F}$ of $\varphi_1^*$ in ${\mathbb P}^n_K$ consists of a finite collection of linear spaces, each one of dimension at most $n-2$. As $X_n^k(\Lambda)$ is not contained inside a hyperplane, it follows that the locus of fixed points of $\varphi_1^*$ in $X_n^k(\Lambda)$, $F=\hat{F}\cap X_n^k(\Lambda)$, must be a finite collection of components each one of dimension at most $d-2$, a contraction. It follows then that $H=H_0$ are required.

\item Let us now assume that either (i) $p=2$ or (ii) $p>2$ and $(d;k,n)\not\in\{(2;2,5),(2;4,3)\}$. By Corollary \ref{Corollary Kontogeorgis igualdad de automorfismos en GFV} we have that $\rm{Aut}(X_n^k(\Lambda))=\rm{Lin}(X_n^k(\Lambda))$. So, from the above,  $\rm{Aut}(X_n^k(\Lambda))$ has a unique generalized Fermat group $H$ of type $(d;k,n)$.

\end{enumerate}



\end{document}